\documentclass[11pt]{article}
\usepackage[margin=1in]{geometry} 
\usepackage{amssymb,amsfonts,amsmath,bbm,mathrsfs,stmaryrd,mathtools}
\usepackage{xcolor}
\usepackage{url}

\usepackage[shortlabels]{enumitem}
\usepackage{lipsum}
\usepackage{tensor}

\usepackage[colorlinks,
linkcolor=black!75!red,
citecolor=blue,
pdftitle={},
pdfproducer={pdfLaTeX},
pdfpagemode=None,
bookmarksopen=true,
bookmarksnumbered=true,
backref=page]{hyperref}

\usepackage{tikz}
\usetikzlibrary{arrows,calc,decorations.pathreplacing,decorations.markings,intersections,shapes.geometric,through,fit,shapes.symbols,positioning,decorations.pathmorphing}

\usepackage{braket}

\usepackage[amsmath,thmmarks,hyperref]{ntheorem}
\usepackage{cleveref}

\creflabelformat{enumi}{#2#1#3}

\crefname{section}{Section}{Sections}
\crefformat{section}{#2Section~#1#3} 
\Crefformat{section}{#2Section~#1#3} 

\crefname{subsection}{\S}{\S\S}
\AtBeginDocument{%
  \crefformat{subsection}{#2\S#1#3}%
  \Crefformat{subsection}{#2\S#1#3}%
}

\crefname{subsubsection}{\S}{\S\S}
\AtBeginDocument{%
  \crefformat{subsubsection}{#2\S#1#3}%
  \Crefformat{subsubsection}{#2\S#1#3}%
}

\theoremstyle{plain}

\newtheorem{lemma}{Lemma}[section]
\newtheorem{proposition}[lemma]{Proposition}
\newtheorem{corollary}[lemma]{Corollary}
\newtheorem{theorem}[lemma]{Theorem}

\theoremstyle{plain}
\theoremnumbering{Alph}
\newtheorem{theoremN}{Theorem}

\theoremstyle{plain}
\theorembodyfont{\upshape}
\theoremsymbol{\ensuremath{\blacklozenge}}

\newtheorem{definition}[lemma]{Definition}
\newtheorem{example}[lemma]{Example}
\newtheorem{examples}[lemma]{Examples}
\newtheorem{remark}[lemma]{Remark}
\newtheorem{remarks}[lemma]{Remarks}

\crefname{definition}{definition}{definitions}
\crefformat{definition}{#2definition~#1#3} 
\Crefformat{definition}{#2Definition~#1#3} 

\crefname{ex}{example}{examples}
\crefformat{example}{#2example~#1#3} 
\Crefformat{example}{#2Example~#1#3} 

\crefname{exs}{example}{examples}
\crefformat{examples}{#2example~#1#3} 
\Crefformat{examples}{#2Example~#1#3} 

\crefname{remark}{remark}{remarks}
\crefformat{remark}{#2remark~#1#3} 
\Crefformat{remark}{#2Remark~#1#3} 

\crefname{remarks}{remark}{remarks}
\crefformat{remarks}{#2remark~#1#3} 
\Crefformat{remarks}{#2Remark~#1#3} 

\crefname{convention}{convention}{conventions}
\crefformat{convention}{#2convention~#1#3} 
\Crefformat{convention}{#2Convention~#1#3} 

\crefname{notation}{notation}{notations}
\crefformat{notation}{#2notation~#1#3} 
\Crefformat{notation}{#2Notation~#1#3} 

\crefname{table}{table}{tables}
\crefformat{table}{#2table~#1#3} 
\Crefformat{table}{#2Table~#1#3}

\crefname{lemma}{lemma}{lemmas}
\crefformat{lemma}{#2lemma~#1#3} 
\Crefformat{lemma}{#2Lemma~#1#3} 

\crefname{proposition}{proposition}{propositions}
\crefformat{proposition}{#2proposition~#1#3} 
\Crefformat{proposition}{#2Proposition~#1#3} 

\crefname{corollary}{corollary}{corollaries}
\crefformat{corollary}{#2corollary~#1#3} 
\Crefformat{corollary}{#2Corollary~#1#3} 

\crefname{theorem}{theorem}{theorems}
\crefformat{theorem}{#2theorem~#1#3} 
\Crefformat{theorem}{#2Theorem~#1#3} 

\crefname{enumi}{}{}
\crefformat{enumi}{#2#1#3}
\Crefformat{enumi}{#2#1#3}

\crefname{assumption}{assumption}{Assumptions}
\crefformat{assumption}{#2assumption~#1#3} 
\Crefformat{assumption}{#2Assumption~#1#3} 

\crefname{construction}{construction}{Constructions}
\crefformat{construction}{#2construction~#1#3} 
\Crefformat{construction}{#2Construction~#1#3} 

\crefname{equation}{}{}
\crefformat{equation}{(#2#1#3)} 
\Crefformat{equation}{(#2#1#3)}

\numberwithin{equation}{section}

\theoremstyle{nonumberplain}
\theoremsymbol{\ensuremath{\blacksquare}}

\newtheorem{proof}{Proof}
\newcommand\pf[1]{\newtheorem{#1}{Proof of \Cref{#1}}}

\newcommand\bA{{\mathbb A}}
\newcommand\bB{{\mathbb B}}
\newcommand\bC{{\mathbb C}}
\newcommand\bD{{\mathbb D}}
\newcommand\bE{{\mathbb E}}
\newcommand\bF{{\mathbb F}}
\newcommand\bG{{\mathbb G}}
\newcommand\bH{{\mathbb H}}

\newcommand\bK{{\mathbb K}}

\newcommand\bQ{{\mathbb Q}}

\newcommand\bS{{\mathbb S}}
\newcommand\bT{{\mathbb T}}

\newcommand\bZ{{\mathbb Z}}

\newcommand\cC{{\mathcal C}}

\DeclareMathOperator{\id}{id}
\DeclareMathOperator{\End}{\mathrm{End}}
\DeclareMathOperator{\Ext}{\mathrm{Ext}}
\DeclareMathOperator{\Hom}{\mathrm{Hom}}
\DeclareMathOperator{\Aut}{\mathrm{Aut}}

\newcommand\spr[1]{\cite[\href{https://stacks.math.columbia.edu/tag/#1}{Tag {#1}}]{stacks-project}}
\newcommand{\qedhere}{\mbox{}\hfill\ensuremath{\blacksquare}}

\newcommand{\xrightarrowdbl}[2][]{%
  \xrightarrow[#1]{#2}\mathrel{\mkern-14mu}\rightarrow
}

\title{Extensible endomorphisms of compact groups}
\author{Alexandru Chirvasitu}

\begin{document}

\date{}

\newcommand{\Addresses}{{
  \bigskip
  \footnotesize

  \textsc{Department of Mathematics, University at Buffalo}
  \par\nopagebreak
  \textsc{Buffalo, NY 14260-2900, USA}  
  \par\nopagebreak
  \textit{E-mail address}: \texttt{achirvas@buffalo.edu}


}}

\maketitle

\begin{abstract}
  We show that the endomorphisms of a compact connected group that extend to endomorphisms of every compact overgroup are precisely the trivial one and the inner automorphisms; this is an analogue, for compact connected groups, of results due to Schupp and Pettet on discrete groups (plain or finite). A somewhat more surprising result is that if $\mathbb{A}$ is compact connected and abelian, its endomorphisms extensible along morphisms into compact connected groups also include $-\mathrm{id}$ (in addition to the obvious trivial endomorphism and the identity). Connectedness cannot be dropped on either side in this last statement. 
\end{abstract}


\vspace{.5cm}



\section*{Introduction}

The automorphisms of a group $\bG$ that extend to automorphisms of every larger group $\bH\ge \bG$ are (as one might have expected), precisely the inner ones: \cite[Theorem]{schupp_inner}. The analogous statements hold \cite[Corollary]{pettet_inner-fin} for other classes of groups (i.e. with $\bG$ and $\bH$ both restricted to the relevant class), such as finite, or finite solvable, or finite nilpotent, perhaps further constrained to orders divisible only by pre-specified sets of primes. The relevant notion is as follows: 

\begin{definition}\label{def:ext}
  A endomorphism $\varphi\in \End_{\cC}(c)$ of an object $c$ in a category $\cC$ {\it extends (or is extensible) along} a morphism $c\xrightarrow{\theta}c'$ if it fits into a commutative diagram
  \begin{equation*}
    \begin{tikzpicture}[auto,baseline=(current  bounding  box.center)]
      \path[anchor=base] 
      (0,0) node (l) {$c$}
      +(2,.5) node (u) {$c$}
      +(2,-.5) node (d) {$c'$}
      +(4,0) node (r) {$c'$}
      ;
      \draw[->] (l) to[bend left=6] node[pos=.5,auto] {$\scriptstyle \varphi$} (u);
      \draw[->] (u) to[bend left=6] node[pos=.5,auto] {$\scriptstyle \theta$} (r);
      \draw[->] (l) to[bend right=6] node[pos=.5,auto,swap] {$\scriptstyle \theta$} (d);
      \draw[->] (d) to[bend right=6] node[pos=.5,auto,swap] {$\scriptstyle \varphi'$} (r);
    \end{tikzpicture}
  \end{equation*}
  We also refer to a map $\varphi'$ as in the preceding diagram as an {\it extension of $\varphi$ along $\theta$}.
\end{definition}

An adjacent theme is taken up in \cite{bg}, which centers around extending endomorphisms {\it functorially} in the above diagram's $\theta$ (regarded as an object of the {\it comma category} \cite[Exercise 3K]{ahs} $c\downarrow\cC$). It is that paper that motivated and provided the scaffolding for \cite{chi_inner_xv1_toappear}, with \cite[Section 2]{chi_inner_xv1_toappear} focused on functorial extensions in the category of compact groups.

The two setting tend to not {\it quite} be comparable, though they frequently almost are: the functorial versions of the results (e.g. \cite[Corollary 3]{bg}) have stronger hypotheses {\it and} somewhat stronger conclusions than the non-functorial counterparts (such as the already-cited \cite[Theorem]{schupp_inner}). In the present note, retaining the compact-group setting, we revert back to ``one-off'' (non-functorial) extensibility in the sense of \Cref{def:ext}. The connected-compact-group version of \cite[Theorem]{schupp_inner} (or rather of a hypothetical analogue thereof concerning {\it endo}morphisms) reads (\Cref{th:conninner}):

\begin{theoremN}
  Let $\bG$ be a compact connected group and $\varphi$ an endomorphism thereof. The following conditions are equivalent. 
  \begin{enumerate}[(a)]
  \item $\varphi$ is either trivial or an inner automorphism. 
  \item $\varphi$ is extensible along every morphism to a compact group.
  \item $\varphi$ is extensible along every embedding into a compact group.
  \item $\varphi$ is extensible along any morphism to a compact Lie group.
  \item $\varphi$ is extensible along any morphism to a compact Lie group of the form $U(n)\rtimes\Gamma$ for finite $\Gamma$ acting on $U(n)$.  \qedhere
  \end{enumerate}
\end{theoremN}

The careful listing of the various mutually equivalent conditions is not (always, or only) pedantry: they genuinely make a difference. Were $\bG$ above abelian for instance, the result would not hold as such for endomorphisms extensible only along maps to {\it connected} compact groups: per \Cref{th:ab-conn}, there is also the perhaps surprising additional option of $-\id$:

\begin{theoremN}
  The endomorphisms of a compact connected group extensible along all morphisms to (or embeddings into) compact connected (Lie) groups are precisely the trivial one and $\pm\id$.  \qedhere
\end{theoremN}

Note incidentally that connectedness on {\it both} sides of the morphisms is essential in a result such as the one just stated: the discussions in \hyperref[exs:adisc]{Examples~\ref*{exs:adisc}} and \Cref{ex:unnotext} each drop one type of connectedness.

As for morphisms whose domain and codomain are {\it both} abelian, the surprises (if any) are not major, some care is needed when handling the profinite case. Compressing \Cref{th:abendo} for brevity:

\begin{theoremN}
  \begin{enumerate}[(1)]
  \item The endomorphisms extensible in the category of compact abelian groups are precisely the scaling operators by integers.

  \item For non-profinite compact abelian groups these are also the morphisms extensible along maps into finite-dimensional tori.

  \item While for profinite groups the latter are the scaling operators by profinite integers, i.e. elements of the profinite completion $\widehat{\bZ}$ of $(\bZ,+)$.  \qedhere
  \end{enumerate}
\end{theoremN}

\subsection*{Acknowledgements}

I am grateful for helpful comments from R. Kanda, I. Penkov, D. Sage, A. Sikora and S. P. Smith.

This work is partially supported by NSF grant DMS-2001128.

\section{Extending compact-group endomorphisms}\label{se:cpct}

For (connected) Lie groups, simplicity is understood as that of the Lie algebra \cite[following Proposition 6.5]{hm4}: the non-trivial normal subgroups are discrete but possibly non-trivial. We refer to the stronger property that the group have no proper non-trivial normal subgroups as {\it algebraic} simplicity. For compact groups, this is also equivalent \cite[Theorem 9.90]{hm4} to asking that {\it closed} proper normal subgroups be trivial. 

One of the main results, announced above in the Introduction, is

\begin{theorem}\label{th:conninner}
  For a compact connected group $\bG$ the following conditions on $\varphi\in \End(\bG)$ are equivalent.
  \begin{enumerate}[(a)]
  \item\label{item:th:conninner-a} $\varphi$ is either trivial or an inner automorphism. 
  \item\label{item:th:conninner-b} $\varphi$ is extensible along every morphism to a compact group.
  \item\label{item:th:conninner-c} $\varphi$ is extensible along every embedding into a compact group.
  \item\label{item:th:conninner-d} $\varphi$ is extensible along any morphism to a compact Lie group.
  \item\label{item:th:conninner-e} $\varphi$ is extensible along any morphism to a compact Lie group of the form $U(n)\rtimes\Gamma$ for finite $\Gamma$ acting on $U(n)$. 
  \end{enumerate}
\end{theorem}

Some small diversions are needed. The following result extends part of \cite[Proposition 2.3]{chi_inner_xv1_toappear} (which in turn proves slightly more than it states). Its aim is to rid us of the mild annoyance of endomorphisms that fail to be bijective. 

\begin{proposition}\label{pr:extembinj}
  For an endomorphism $\varphi\in\End(\bG)$ of a compact group, consider the following conditions:
  \begin{enumerate}[(a)]

  \item\label{item:embinj} $\varphi$ is extensible along every embedding into a compact connected group.
    
  \item\label{item:morbij} $\varphi$ is extensible along any morphism to a compact connected (algebraically simple) Lie group.
    
  \item\label{item:isbij} $\varphi$ is either trivial or an automorphism. 
    
  \end{enumerate}
  We then have
  \begin{equation*}
    \text{
      \Cref{item:embinj} $\xRightarrow{\quad}$ \Cref{item:morbij} $\xRightarrow{\quad}$ \Cref{item:isbij}
    }.
  \end{equation*}
  Similarly, dropping connectedness in both \Cref{item:embinj} and \Cref{item:morbij}, the former implies the other. 
\end{proposition}
\begin{proof}
  Regarding \Cref{item:embinj} and \Cref{item:morbij}, we focus on the connected version of the result; the other admits a virtually identical proof. 
  
  \begin{description}[style=unboxed, leftmargin=*, wide=0pt]
    
  \item{\bf \Cref{item:embinj} $\xRightarrow{\quad}$ \Cref{item:morbij}:} The claim is that \Cref{item:morbij} follows from the {\it weakest} version of \Cref{item:embinj}. To see this, consider a morphism $\bG\xrightarrow{\theta}\bH$ to a compact connected Lie group. On the one hand $\bG$ embeds into a product of compact connected Lie groups \cite[Corollary 2.36]{hm4}, while on the other, every compact Lie group embeds \cite[Lemma 2.4]{chi_inner_xv1_toappear} into a compact, connected, algebraically simple Lie group. This means that there is an embedding $\bG\lhook\joinrel\xrightarrow{\iota}\bK$ into a product of compact, connected, algebraically simple Lie groups, all of dimension larger than $\dim \bH$.
    
    The hypothesis ensures the extensibility of $\varphi$ along the {\it embedding}
    \begin{equation}\label{eq:thetapi}
      \bG\xrightarrow{\quad (\theta,\iota)\quad}\bH\times \bK. 
    \end{equation}
    Because the individual simple factors of $\bK$ have dimension larger than $\dim \bH$, they only map to the latter trivially. This means that every endomorphism of $\bH\times \bK$ leaves the factor $\bK$ invariant, so that {\it every} endomorphism of $\bH\times \bK$ extends along the first projection $\bH\times \bK\xrightarrow{\pi_1} \bH$. Composing the two extensions of $\varphi$ (first along \Cref{eq:thetapi} and then along $\pi_1$), we obtained the desired extensibility along $\theta$.

  \item{\bf \Cref{item:morbij} $\xRightarrow{\quad}$ \Cref{item:isbij}:} We prove the strictest version of the claim. Assuming $\varphi$ is not trivial, its injectivity follows from \cite[Proposition 2.3, part (1) of the proof]{chi_inner_xv1_toappear}. The argument for surjectivity, on the other hand, is a slight modification of \cite[Proposition 2.3, part (2) of the proof]{chi_inner_xv1_toappear}.
   
    Assume that $\varphi$ is {\it not} onto (and also non-trivial, i.e. $\bG\ne \{1\}$). We can then find a morphism $\bG\xrightarrow{\iota}\bH$ into a compact connected algebraically simple Lie group with
    \begin{equation*}
      \{1\}
      \lneq
      \iota(\varphi(\bG))
      \lneq
      \iota(\bG)\le \bH. 
    \end{equation*}
    In the commutative diagram
    \begin{equation*}
      \begin{tikzpicture}[auto,baseline=(current  bounding  box.center)]
        \path[anchor=base] 
        (0,0) node (l) {$\bG$}
        +(2,.5) node (u) {$\bG$}
        +(2,-.5) node (d) {$\bH$}
        +(4,0) node (r) {$\bH$}
        ;
        \draw[->] (l) to[bend left=6] node[pos=.5,auto] {$\scriptstyle \varphi$} (u);
        \draw[->] (u) to[bend left=6] node[pos=.5,auto] {$\scriptstyle \iota$} (r);
        \draw[->] (l) to[bend right=6] node[pos=.5,auto,swap] {$\scriptstyle \iota$} (d);
        \draw[->] (d) to[bend right=6] node[pos=.5,auto,swap] {$\scriptstyle \varphi'$} (r);
      \end{tikzpicture}
    \end{equation*}
    the top path is non-trivial and hence so is $\varphi'$. But then the latter must be an {\it auto}morphism of $\bH$ (injective by algebraic simplicity and onto for dimension reasons), so it cannot shrink $\iota(\bG)$ properly into
    \begin{equation*}
      \iota(\varphi(\bG)) = \varphi'(\iota(\bG)).
    \end{equation*}
    This gives the desired contradiction, finishing the proof.    
  \end{description}
\end{proof}

\begin{remark}\label{re:whydifferent}
  \cite[Proposition 2.3, part (2) of the proof]{chi_inner_xv1_toappear} at one point relies on the {\it functorial} extensibility assumed there, hence the need for the slight alteration in the argument above.
\end{remark}


As part of a series of remarks that will eventually plug back into the proof of \Cref{th:conninner}, note that requiring extensibility along morphisms into possibly {\it dis}connected compact groups is crucial. \Cref{le:aut-1} provides a first inkling of this. Recall \cite[Definitions 9.30]{hm4} that a {\it pro-torus} is a compact, connected, abelian group. 

\begin{lemma}\label{le:aut-1}
  Given a pro-torus $\bA\le \bG$ of a compact, connected group, there is an automorphism of $\bG$ that restricts to $-\id$ on $\bA$.
\end{lemma}
\begin{proof}
  Per standard structure theory \cite[Theorem 9.24]{hm4}, $\bG$ is a quotient of a compact product
  \begin{equation*}
    \bT\times \prod_{i\in I}\bS_i
    ,\quad
    \bT\text{ a pro-torus}
    ,\quad
    \bS_i\text{ simple, connected, simply-connected Lie}
  \end{equation*}
  by a 0-dimensional central subgroup. The image of $\bA$ in each of the factors $\bS_i$ is a torus (being compact, connected, abelian and Lie). Because maximal tori in $\bS_i$ (being their own centralizers \cite[Theorem 9.32 or Corollary 6.33]{hm4}) respectively contain the centers of $\bS_i$, it is enough to argue that for a compact, connected, simple Lie group $\bS$ there is an automorphism restricting to $-\id$ on any given maximal torus $\bT\subset \bS$.
  
  We may as well assume without loss of generality that $\bS$ is simply-connected, whereupon the claim is a consequence of its complex-Lie-algebra analogue \cite[Proposition 14.3]{hmph_liealg}: a complex, simple Lie algebra has an automorphism acting as $-\id$ on any given Cartan subalgebra. 
\end{proof}

\begin{remarks}\label{res:posttor-1}
  \begin{enumerate}[(1), leftmargin=*, wide=0pt]
  \item As already hinted at, it is natural to wonder to what extent connectedness is crucial to \Cref{le:aut-1}. Since the only automorphisms of a finite group extensible along embeddings into finite groups are inner \cite[Corollary]{pettet_inner-fin}, the two connectedness constraints can certainly not {\it both} be dropped. In fact, we will see that in fact neither condition can be dropped in isolation (while retaining the full generality of \Cref{le:aut-1} otherwise): \hyperref[exs:adisc]{Examples~\ref*{exs:adisc}} address the (dis)connected of $\bA$, while \Cref{ex:unnotext} is concerned with that of $\bG$.
    
  \item Once the setup of \Cref{le:aut-1}, the requisite automorphism of $\bG$ can even be chosen inner, so long as $-1$ belongs to the Weyl group. This happens \cite[\S 13.4, Exercise 5]{hmph_liealg} in types $A_1$, $B_{\ell}$, $C_{\ell}$, $D_{2k}$, $E_7$, $E_8$, $F_4$ and $G_2$.
  \end{enumerate}
\end{remarks}

The subgroup $\bA\le \bG$ of \Cref{le:aut-1} cannot in general be assumed disconnected, even if $\bG$ is connected.

\begin{examples}\label{exs:adisc}
  There compact, simply-connected Lie groups containing (necessarily disconnected) maximal abelian subgroups whose inversion automorphism $x\xmapsto{}-x$ does not lift to the ambient group.
  \begin{enumerate}[(1), leftmargin=*, wide=0pt]
  \item {\bf : Type $F_4$.} If $\bG$ is the simple compact Lie group of type $F_4$, there is a maximal abelian subgroup \cite[equation (75)]{yu_maxab}
    \begin{equation*}
      (\bZ/3)^3\cong \bA\subset \bG
    \end{equation*}
    whose Weyl group (i.e. quotient of the normalizer) consists precisely of the {\it unimodular} $3\times 3$ matrices over the field $\bF_3$ \cite[Proposition 3.14]{yu_maxab} (see also \cite[Theorem 5.22]{ek_liegradings}). In particular, the inversion automorphism of $\bA$ does not extend to an inner automorphism of $\bG$; but {\it all} automorphisms are inner \cite[p.953]{afg_max}.

  \item\label{item:exs:adisc-e6} {\bf : Type $E_6$.} Let $\bG$ be compact, simply-connected of type $E_6$. Its {\it adjoint} quotient $\bG_{ad}:=\bG/Z(\bG)$ (by the order-3 center: see the list immediately preceding \cite[\S 13.2]{hmph_liealg}) has an abelian subgroup
    \begin{equation*}
      (\bZ/3)^3\cong \overline{\bA}\subset \bG_{ad}
    \end{equation*}
    which, when lifted to $\bG$, generates a maximal abelian (finite) subgroup $\bA$ together with the center $Z(\bG)$:
    \begin{equation*}
      \bA:=\text{preimage of $\overline{\bA}$ through }\bG\xrightarrowdbl{\quad}\bG/Z(\bG). 
    \end{equation*}
    The group $\overline{\bA}$ is the one recorded in \cite[$F_2$, Table 4, second row]{yu_maxab}, and its Weyl group is again $SL_3(\bF_3)$ \cite[Proposition 4.4]{yu_maxab}. It thus follows that no inner automorphism of $\bG_{ad}$ can operate on $\overline{\bA}$ as $-\id$. Since three generators of $\overline{\bA}$ lift to a {\it rank-0 commuting triple} \cite[\S 1.9]{bfm_almostcomm} in $\bG$ and the latter's automorphisms operate on such triples as conjugations \cite[Proposition 5.3.1]{bfm_almostcomm}, it follows that no automorphism of $\bG$ can invert $\bA$. 
    
  \item {\bf : Type $E_8$.} The only compact Lie group of type $E_8$ is simply-connected (and center-less, e.g. by the selfsame \cite[discussion preceding \S 13.2]{hmph_liealg}), so no issues of lifting arise.

    The group listed in \cite[$F_1$, Table 7, row 1]{yu_maxab} is maximal abelian and isomorphic to $(\bZ/5)^3$, with Weyl group $SL_3(\bF_5)$ \cite[Proposition 7.2]{yu_maxab} (hence excluding $-\id$). Because the automorphisms of $E_8$ are all inner (e.g. by \cite[\S 16.5]{hmph_liealg}; the type-$E_8$ Dynkin diagram has no non-trivial automorphisms), the conclusion follows. 
  \end{enumerate}
\end{examples}

\begin{remark}\label{re:e6}
  \cite[Propositions 1.28 and 1.32]{ek_liegradings} explain how (possibly disconnected) maximal abelian subgroups of adjoint simple compact Lie groups relate to the {\it fine gradings} \cite[\S 1.3]{ek_liegradings} of the corresponding simple Lie algebras (the gradings are by the discrete abelian {\it Pontryagin duals} \cite[\S 3.5]{de} of the respective maximal groups).

  The literature on fine gradings is vast (as the above monograph's reference list will confirm). So far as type $E_6$ goes, the full list of such gradings appears, for instance, in \cite[Figure 6.2, p.268, first column]{ek_liegradings} or \cite[Theorem 1, table]{ev_finegre6}. Note, though, that those lists do {\it not} contain the group denoted by $\overline{\bA}$ in \Cref{exs:adisc}\Cref{item:exs:adisc-e6}: the reason is that the class of subgroups of a compact simple Lie group classified in \cite{yu_maxab} (those satisfying \cite[Introduction, condition (*)]{yu_maxab}) includes \cite[\S 2.5]{yu_maxab} that of maximal subgroups {\it strictly}: $\overline{\bA}$ is not maximal in the adjoint compact form of $E_6$, but its lift to the simply-connected form is maximal therein.
\end{remark}

We next need what must be a straightforward observation, added here for completeness and because I have been unable to find it in the literature (not in the desired form, ready-citable). Following standard practice (e.g. \cite[Definition preceding Theorem 9.70]{rot_halg}), we refer to a group morphism
\begin{equation}\label{eq:projrep}
  \Gamma
  \xrightarrow{\quad\rho\quad}
  \bG/Z(\bG)
\end{equation}
as a {\it projective representation} of $\Gamma$ in $\bG$ ($\Gamma$ will typically be discrete). $\rho$ gives rise (\cite[proof of Theorem 9.70]{rot_halg} or \cite[\S VII.3]{ser_locf}) to a {\it 2-cocycle}
\begin{equation}\label{eq:correspcocyc}
  z=z_{\rho}\in Z^2(\Gamma,\ Z(\bG))
  ,\quad
  z(\gamma,\gamma'):=\widetilde{\rho}(\gamma)\widetilde{\rho}(\gamma')\widetilde{\rho}(\gamma\gamma')^{-1}
\end{equation}
for a lift $\Gamma\xrightarrow{\widetilde{\rho}}\bG$ of $\rho$ (the $\rho\xleftrightarrow{\quad}\widetilde{\rho}$ notation will recur). In parsing the following statement, note that a projective representation \Cref{eq:projrep} induces an action of $\Gamma$ on $\bG$ by conjugation:
\begin{equation*}
  \bG\ni x\xmapsto{\quad \gamma\quad}
  \tensor[^{\gamma}]{x}{}
  :=
  \widetilde{\rho}(\gamma)x\widetilde{\gamma}(\gamma)^{-1}. 
\end{equation*}
It thus makes sense to speak of the resulting semidirect product $\bG\rtimes \Gamma$. 

\begin{lemma}\label{le:whenlifts}
  Let \Cref{eq:projrep} be a projective representation with corresponding cocycle \Cref{eq:correspcocyc}.

  An automorphism $\alpha$ of $\bG$ lifts to $\bG\rtimes \Gamma$ if and only if the cohomology class of $\alpha(z)$ is in the orbit of $\Aut(\Gamma)$ acting on $H^2(\Gamma,\ Z(\bG))$. 
\end{lemma}
\begin{proof}
  An extension of $\alpha$ to $\bG\rtimes \Gamma$, restricted to $\Gamma$, will be of the form
  \begin{equation*}
    \Gamma\ni \gamma
    \xmapsto{\quad}
    a_{\gamma}\theta(\gamma)
    \in \bG\rtimes \Gamma
  \end{equation*}
  for
  \begin{itemize}
  \item an automorphism $\theta\in \Aut(\Gamma)$;
  \item and a (non-abelian) 1-cocycle
    \begin{equation*}
      (a_{\gamma})_{\gamma}\in Z^1(\Gamma,\ \bG),
    \end{equation*}
    where the action of $\Gamma$ on $\bG$ pertinent to the cocycle space is the original one, twisted by $\theta$:
    \begin{equation*}
      \bG\ni x
      \xmapsto{\quad\gamma\quad}
      \tensor[^{\theta\gamma}]{x}{}
      =
      \widetilde{\rho}(\theta\gamma)x\widetilde{\rho}(\theta\gamma)^{-1}. 
    \end{equation*}
    This is a version of the correspondence \cite[\S I.5.1, Exercise 1]{ser_galcoh} between 1-cocycles and lifts of the surjection $\bG\rtimes\Gamma\to \Gamma$; the requisite cocycle condition is
    \begin{equation}\label{eq:aiscocyc}
      a_{\gamma\gamma'}
      =
      a_{\gamma}\cdot \tensor[^{\theta\gamma}]{a}{_{\gamma'}}
      ,\quad
      \forall \gamma,\ \gamma'\in \Gamma.
    \end{equation}        
  \end{itemize}
  The compatibility condition between $\alpha$ and the desired extension (after some processing) reads
  \begin{equation*}
    a_{\gamma}c_{\gamma} = \alpha\left(\widetilde{\rho}\gamma\right)\cdot \widetilde{\rho}\left(\theta \gamma\right)^{-1}
    ,\quad
    \forall \gamma\in \Gamma
    ,\quad \text{where }c_{\gamma}\in Z(\bG). 
  \end{equation*}
  The cocycle constraint \Cref{eq:aiscocyc} now simply means that the 2-cocycle
  \begin{equation*}
    \left(\Gamma^2\ni (\gamma,\gamma')
    \xmapsto{\quad}
    \alpha(z(\gamma,\gamma'))\cdot z(\theta\gamma,\theta\gamma')^{-1}\right)
  \in Z^2(\Gamma,\ Z(\bG))
  \end{equation*}
  is the coboundary of the 1-cochain $(c_{\gamma})_{\gamma}$. In short, finding such $(c_{\gamma})_{\gamma}$ and $\theta$ is equivalent to the class in $H^2(\Gamma,\ Z(\bG))$ of $\alpha(z)$ coinciding with that of $z$ twisted by $\theta\in \Aut(\Gamma)$. 
\end{proof}

\begin{corollary}\label{cor:whenextun}
  Let $\Gamma\xrightarrow{\rho}PU(n)$ be a projective representation, $z\in Z^2(\Gamma,\ Z(U(n)))$ the corresponding 2-cocycle \Cref{eq:correspcocyc}, and $\overline{z}\in H^2(\Gamma,\ Z(U(n)))$ its cohomology class.

  The following conditions are equivalent:
  \begin{enumerate}[(a)]
  \item\label{item:allautos} Every automorphism of $U(n)$ lifts to an automorphism of $U(n)\rtimes\Gamma$.

  \item\label{item:allouterautos} All outer automorphisms of $U(n)$ lift to automorphisms of $U(n)\rtimes\Gamma$.

  \item\label{item:1outerauto} Any one specific outer automorphism of $U(n)$ (e.g. complex conjugation) lifts to $U(n)\rtimes\Gamma$.

  \item\label{item:z-1} $\overline{z}^{-1}$ is in the orbit of $\overline{z}$ under $\Aut(\Gamma)$. 
  \end{enumerate}
\end{corollary}
\begin{proof}
  The equivalence \Cref{item:allautos} $\xLeftrightarrow{\quad}$ \Cref{item:allouterautos} of course holds in complete generality, since {\it inner} automorphisms of $U(n)$ certainly extend to any overgroup. That this is all further equivalent to \Cref{item:1outerauto} follows from the fact that
  \begin{equation*}
    \Aut(U(n))\cong PU(n)\rtimes \bZ/2
    ,\quad
    \bZ/2\text{ acting by complex conjugation}
  \end{equation*}
  (analogous to the classification of automorphisms for the type-$A$ simple Lie algebras, e.g. \cite[\S IX.5, Theorem 5]{jc}). 

  Finally, \Cref{item:1outerauto} $\xLeftrightarrow{\quad}$ \Cref{item:z-1} follows from \Cref{le:whenlifts} and the fact that complex conjugation operates on $Z(U(n))$ as inversion. 
\end{proof}

Returning, now, to the failure of \Cref{le:aut-1} for disconnected $\bG$:

\begin{example}\label{ex:unnotext}
  We consider the setting of \Cref{cor:whenextun}, with finite $\Gamma$. Every cohomology class in
  \begin{equation*}
    H^2(\Gamma,\ \bC^{\times})\cong H^2(\Gamma,\ \bS^1)\cong H^2(\Gamma,\ \bQ/\bZ)
  \end{equation*}
  is realizable via a projective representation into some $PU(n)$ (consider a central extension of $\Gamma$ by some finite subgroup of $\bQ/\bZ$ corresponding \cite[\S 5.3, Theorem 1]{gruen_coh} to said cohomology class, represent that extension in some $U(n)$, etc.). Because inner automorphisms act trivially on cohomology \cite[\S VII.5, Proposition 3]{ser_locf}, it will be enough to produce a group $\Gamma$ with
  \begin{itemize}
  \item only inner automorphisms;
  \item and a cohomology class $\overline{z}\in H^2(\Gamma,\ \bC^{\times})$ of order $>2$, so that its inverse is not in the (singleton!) orbit of $\overline{z}$ under $\Aut(\Gamma)$. 
  \end{itemize}
  This is not difficult to do, with $\Gamma$ {\it metabelian}, for instance (i.e. an extension of an abelian group by another). The metabelian groups of order $>2$ with only inner automorphisms are \cite[Theorem 1]{gr_metabinner} precisely those of the form
  \begin{equation*}
    \prod_{i}\left(\bZ/p_i^{n_i},+\right)\rtimes \left((\bZ/p_i^{n_i})^{\times},\cdot\right),
  \end{equation*}
  where the $p_i^{n_i}$ are finitely many distinct odd prime powers and
  \begin{equation*}
    \left((\bZ/p_i^{n_i})^{\times},\cdot\right)\cong \Aut\left(\bZ/p_i^{n_i}, +\right).
  \end{equation*}
  Consider, say, a product of two such factors:
  \begin{equation*}
    \Gamma =
    (\bZ/9,+)\rtimes (\bZ/6,+)
    \times
    (\bZ/27,+)\rtimes (\bZ/18,+).
  \end{equation*}
  The quotient $\bZ/6\times \bZ/18$ of $\Gamma$ is split, so
  \begin{equation}\label{eq:lhs}
    H^2(\bZ/6\times \bZ/18,\ \bC^{\times})\cong \bZ/6
  \end{equation}
  is a summand in $H^2(\Gamma,\ \bC^{\times})$ (\Cref{eq:lhs} is a simple computation via Lyndon-Hochschild-Serre \cite[Theorem 10.52]{rot_halg} say, given the well-known cohomology of finite cyclic groups \cite[Theorem 9.27]{rot_halg}). In particular, the latter group has non-involutive elements.
\end{example}

It is perhaps worth noting explicitly what types of uses one can put \Cref{ex:unnotext} to.

\begin{lemma}\label{le:urtimesgamma}
  For every positive integer $d$, there are compact Lie groups of the form $U(n)\rtimes\Gamma$ with $n$ divisible by $d$ to which no outer automorphism of $U(n)$ lifts. 
\end{lemma}
\begin{proof}
  Once we have a projective representation affording an appropriate cocycle as in \Cref{ex:unnotext}, one can simply replace it with a sum of $d$ copies thereof. 
\end{proof}

\pf{th:conninner}
\begin{th:conninner}
  The downstream implications
  \begin{equation*}
    \text{
      \Cref{item:th:conninner-a}
      $\xRightarrow{\ }$
      \Cref{item:th:conninner-b}
      $\xRightarrow{\ }$
      \Cref{item:th:conninner-c}
      $\xRightarrow{\ }$
      \Cref{item:th:conninner-d}
      $\xRightarrow{\ }$
      \Cref{item:th:conninner-e}
    }
  \end{equation*}
  are self-evident, so addressing \Cref{item:th:conninner-e} $\xRightarrow{\ }$ \Cref{item:th:conninner-a} will suffice. We assume $\varphi$ non-trivial so that $\bG$ is non-trivial and (\Cref{pr:extembinj}) $\varphi$ is an automorphism. Consider unitary representations
  \begin{equation*}
    \begin{aligned}
      \bG&\xrightarrow{\quad \rho'\quad}U(d')\text{ non-trivial, irreducible, arbitrary and}\\
      \bG&\xrightarrow{\quad \rho\quad}U(d)\text{ defined by}
           \begin{cases}
             \rho=\rho'&\text{if }d'>1\\
             \rho=\left(\rho'\right)^{\oplus d}&\text{otherwise},
           \end{cases}
    \end{aligned}
  \end{equation*}
  where in the latter case $d$ is chosen so that the image $\varphi(\rho')$ of $\rho'$ under the action of $\varphi\in\Aut(\bG)$ on irreducible representations is {\it not} a power of $(\rho')^{\otimes d}$ (not a typo: that is a tensor power of a 1-dimensional representation, hence another such). 

  Next, consider the morphism

  \begin{equation*}
    \begin{tikzpicture}[auto,baseline=(current  bounding  box.center)]
      \path[anchor=base] 
      (0,0) node (l) {$\bG$}
      +(3,.5) node (u) {$U(n)$}
      +(5,0) node (r) {$U(n)\rtimes\Gamma$}
      ;
      \draw[->] (l) to[bend left=6] node[pos=.5,auto] {$\scriptstyle \rho^{\oplus (n/d)}$} (u);
      \draw[right hook->] (u) to[bend left=6] node[pos=.5,auto] {$\scriptstyle $} (r);
      \draw[->] (l) to[bend right=6] node[pos=.5,auto,swap] {$\scriptstyle \theta$} (r);
    \end{tikzpicture}
  \end{equation*}
  for $d|n$ and $\Gamma$ as in \Cref{le:urtimesgamma}. The choices made above ensure that an extension of $\varphi$ along $\theta$, when restricted to $U(n)$, would have to be an automorphism. That automorphism cannot be outer because of the choice of semidirect product via \Cref{le:urtimesgamma}. This means that (the isomorphism class of) $\rho^{\oplus(n/d)}$ and hence also $\rho'$ is fixed by $\varphi$; since $\rho'$ was an arbitrary irreducible representation of the {\it connected} compact group $\bG$, $\varphi$ is inner by \cite[Corollary 2]{mcm-dual}.
\end{th:conninner}


The discussion preceding \Cref{le:aut-1} alludes to possible ``misbehavior'' in \Cref{th:conninner} when imposing extensibility only along morphisms into connected groups. \Cref{le:aut-1} itself suggests that these phenomena might obtain when the base group is abelian, and the following result spells this out in full. 

\begin{theorem}\label{th:ab-conn}
  The following conditions on an endomorphism $\varphi\in\End(\bA)$ of a pro-torus are equivalent:
  \begin{enumerate}[(a)]

  \item\label{item:ab-trivpm1} $\varphi$ is trivial or $\pm\id$. 

  \item\label{item:ab-extmor} $\varphi$ extends along any morphism $\bA\xrightarrow{}\bH$ to a compact connected group.

  \item\label{item:ab-extemb} $\varphi$ extends along any embedding $\bA\lhook\joinrel\xrightarrow{}\bH$ into a compact connected group.
    
  \item\label{item:ab-extmor-lie} $\varphi$ extends along any morphism $\bA\xrightarrow{}\bH$ to a compact connected (algebraically simple) Lie group.

  \end{enumerate}
\end{theorem}
\begin{proof}
  \begin{description}[style=unboxed, leftmargin=*, wide=0pt]

  \item {\bf \Cref{item:ab-trivpm1} $\xRightarrow{\quad}$ \Cref{item:ab-extmor}:} Only the case $\varphi=-\id$ is interesting, and it is delivered by \Cref{le:aut-1}. 

  \item {\bf \Cref{item:ab-extmor} $\xRightarrow{\quad}$ \Cref{item:ab-extemb}} formally.

  \item {\bf \Cref{item:ab-extemb} $\xRightarrow{\quad}$ \Cref{item:ab-extmor-lie}:} This is part of \Cref{pr:extembinj}. 
    
  \item {\bf \Cref{item:ab-extmor-lie} $\xRightarrow{\quad}$ \Cref{item:ab-trivpm1}:} Indeed, $\varphi$ is either trivial or an automorphism by \Cref{pr:extembinj}, as well as scaling by an integer by \Cref{le:torendo}. The only options are to scale by $0$ or $\pm 1$.  \qedhere

  \end{description}
\end{proof}

\section{Extensibility in the category of compact abelian groups}\label{se:cpctab}

The ensuing discussion assumes some background on {\it profinite} groups \cite[\S 2.1]{rz_prof}, i.e. {\it cofiltered (or inverse, or projective) limits} (\cite[\S 1.1]{rz_prof} or \spr{04AY}) of finite quotient groups. These are also precisely \cite[Theorem 1.34]{hm4} the totally disconnected compact groups, and, when abelian, those whose Pontryagin duals are torsion \cite[Corollary 8.5]{hm4}.

Torsion abelian groups being modules in an obvious fashion over the {\it profinite completion} \cite[Example 2.1.6(2)]{rz_prof}
\begin{equation*}
  \widehat{\bZ}:=\varprojlim_n \bZ/n,\quad n\in \bZ_{>0}\text{ ordered by divisibility},
\end{equation*}
their Pontryagin duals (i.e. the profinite abelian groups) are similarly modules over $\widehat{\bZ}$ \cite[discussion preceding \S 5.2]{rz_prof}. In other words: the usual integer-scaling endomorphisms 
\begin{equation}\label{eq:scale}
  \bA\ni a\xmapsto{\quad}na\in \bA,\quad n\in \bZ,
\end{equation}
of abelian groups $\bA$, extend to actions of the profinite ring $\widehat{\bZ}\supset \bZ$ when $\bA$ is profinite. What is more, we have an isomorphism \cite[Example 2.3.11]{rz_prof}
\begin{equation*}
  \widehat{\bZ}\xrightarrow[\cong]{\qquad}\prod_{\text{primes }p}\bZ_p
\end{equation*}
with
\begin{equation*}
  \bZ_p:=\varprojlim_n \bZ/p^n=\text{the ring of {\it $p$-adic integers} \cite[Example 2.1.6(2)]{rz_prof}}
\end{equation*}
(e.g. as a particular instance of the decomposition \cite[Proposition 2.3.8]{rz_prof} of any profinite pro-nilpotent group as the product of its {\it $p$-Sylow subgroups}). $\widehat{\bZ}$ acts on {\it $p$-primary} torsion groups (i.e. \cite[preceding Theorem 1]{kap_infab} those whose elements have $p$-power orders) via its $\bZ_p$ quotient.

\begin{theorem}\label{th:abendo}
  Let $(\bA,+)$ be a compact abelian group and $\varphi\in\End(\bA)$.
  \begin{enumerate}[(1)]
    
  \item\label{item:abendo-all} $\varphi$ is extensible along morphisms into arbitrary compact abelian groups if and only if it is an integer-scaling map \Cref{eq:scale}.
    
  \item\label{item:abendo-tor} If the identity component $\bA_0$ is non-trivial (i.e. $\bA$ is not profinite), $\varphi$ is extensible along morphisms into (finite-dimensional) tori if and only if it is an integer-scaling map \Cref{eq:scale}.

  \item\label{item:abendo-prof} If $\bA$ is profinite, $\varphi$ extends along morphisms to (finite-dimensional) tori if and only it is scaling by some $n\in\widehat{\bZ}$.

  \item\label{item:abendo-protor} The conclusion of \Cref{item:abendo-tor} also holds for profinite $\bA$ if it either
    \begin{itemize}
    \item has a non-trivial torsion-free quotient;
    \item or infinitely many Sylow subgroups. 
    \end{itemize}

  \end{enumerate}
\end{theorem}

It will be convenient to spell out some auxiliary partial results; first a simple observation:

\begin{lemma}\label{le:presallquot}
  Consider the following conditions on an endomorphism $\varphi\in\End(\bA)$ of a compact abelian group.
  \begin{enumerate}[(a)]
  \item\label{item:extor} $\varphi$ extends along any morphism into a finite-dimensional torus.

  \item\label{item:excirc} $\varphi$ extends along any morphism into $\bS^1$. 

  \item\label{item:allin} $\varphi$ leaves invariant every closed subgroup of $\bA$, and hence induces an endomorphism on every quotient of $\bA$.

  \item\label{item:fininv} $\varphi$ leaves invariant every finite-codimensional closed subgroup of $\bA$, and hence induces an endomorphism on every Lie quotient of $\bA$.
  \end{enumerate}
  We have
  \begin{equation*}
    \text{
      \Cref{item:extor}
      $\xRightarrow{\quad}$
      \Cref{item:excirc}
      $\xLeftrightarrow{\quad}$
      \Cref{item:allin}
      $\xLeftrightarrow{\quad}$
      \Cref{item:fininv}
    }.
  \end{equation*}
\end{lemma}
\begin{proof}
  That \Cref{item:extor} implies \Cref{item:excirc} is obvious. As for the rest, it will become transparent upon (Pontryagin-)dualizing. Extensibility of $\varphi$ along $\bA\xrightarrow{\theta}\bB$, for instance, can also be phrased as extensibility of the Pontryagin dual endomorphism $\widehat{\varphi}\in\End(\widehat{\bA})$ along the morphism 
  \begin{equation*}
    \widehat{\bB}\xrightarrow{\quad\widehat{\theta}\quad}\widehat{\bA}
  \end{equation*}
  of {\it discrete} abelian groups in the obvious sense, dual to that of \Cref{def:ext}. With that in place, \Cref{item:excirc} says that $\widehat{\varphi}$ extends along every morphism $\bZ\to\widehat{\bA}$. Since $\End(\bZ)=\bZ$, this means exactly that every element of $\widehat{\bA}$ is mapped into an integer multiple of itself, or that all (cyclic, or finitely-generated, or unrestricted) subgroups of $\widehat{\bA}$ are $\widehat{\varphi}$-invariant. 
\end{proof}

\begin{lemma}\label{le:torendo}
  \Cref{th:abendo} holds for pro-tori (i.e. when $\bA=\bA_0$). 
\end{lemma}
\begin{proof}
  One direction is obvious (scaling maps, of course, extend along morphisms of abelian groups), so we will be focusing on the converse claims. 
  
  The claim is immediate for a {\it torus} $\bA=\bT^d=(\bS^1)^d$: those endomorphisms of its Pontryagin dual \cite[Examples 3.5.1]{de} $\widehat{\bT^d}\cong \bZ^d$ are elements of
  \begin{equation*}
    \End(\bZ^d)\cong M_d(\bZ):=\text{$d\times d$ matrices with integer entries}
  \end{equation*}
  preserving {\it every} subgroup of $\bZ^d$. That such a morphism is scalar (i.e. multiplication by some $n$) then follows immediately. Now write
  \begin{equation*}
    \bA\cong \varprojlim_{i\in I}\bT_i\quad\left(\text{cofiltered limit of tori by \cite[Corollary 2.36]{hm4}}\right),
  \end{equation*}
  apply the preceding observation to each individual torus $\bT_i$, and note that $n\in \bZ$ is uniquely determined by its scaling action on any torus. This means that there must be {\it one} $n$ by which $\varphi$ operates on each of the quotients $\bA\xrightarrowdbl{}\bT_i$.
\end{proof}

\begin{remark}
  Which morphisms one seeks to extend along (e.g. embedding or arbitrary, or morphisms into specific classes of groups) can make a big difference to the classification, as should already be clear from comparing \Cref{th:ab-conn} and \Cref{le:torendo}, say. Other examples of such drastic differences (to the preceding results):
  \begin{enumerate}[(a), leftmargin=*, wide=0pt]
  \item Tori $\bT^I:=(\bS^1)^I$ (for possibly infinite index sets $I$) being {\it injective objects} \cite[Theorem 8.78]{hm4} in the category of compact abelian groups, their embeddings therein split. It follows immediately that {\it all} endomorphisms of a torus extend along its embeddings into compact abelian groups. 

  \item By the same token (injectivity of tori) every endomorphism of every compact abelian group extends along every embedding into a torus.     
  \end{enumerate}
\end{remark}

Because \Cref{th:abendo}\Cref{item:abendo-protor} centers around distinguishing between scaling by profinite integers and honest integers, it will be helpful to isolate the relevant instances of that problem, focusing on particularly pleasant classes of groups.

\begin{proposition}\label{pr:red2dirsums}
  Let $\bA$ be a profinite abelian group and $n\in\widehat{\bZ}$, inducing a scaling endomorphism $\varphi=\varphi_n$ on $\bA$.

  $\varphi$ is in fact induced by an integer if and only if this is so for the corestrictions of $\varphi$ to countable quotients
  \begin{equation*}
    \bA\xrightarrowdbl{\quad}\prod_{i\in\aleph_0}\bA_i,
  \end{equation*}
  where $\bA_i$ are either finite cyclic or copies of the $p$-adic groups $(\bZ_p,+)$ \cite[Example 2.1.6(2)]{rz_prof} for various primes $p$. 
\end{proposition}
\begin{proof}
  There is no need to address one of the implications, so assume $\varphi$ is {\it not} an integer scaling; or, for convenience, we say that $\varphi$ scales by a {\it proper} profinite integer.

  Switching perspective to the Pontryagin dual $\widehat{\bA}$, the claim is one regarding subgroups thereof expressible as countable direct {\it sums} (rather than products): of finite cyclic groups and {\it divisible} \cite[\S 5]{kap_infab} indecomposable summands of the form
  \begin{equation*}
    \bZ/p^{\infty}:=\bigcup_{n\in \bZ_{>0}}\bZ/p^n\cong \left\{\text{$p$-power-order roots of unity}\right\}\subset \bS^1
  \end{equation*}
  (the $Z(p^{\infty})$ of \cite[\S 2 (g)]{kap_infab}), respectively dual to $\bZ_p$. We have \cite[Theorems 3 and 4]{kap_infab}
  \begin{equation}\label{eq:divdec}
      \widehat{\bA}\cong \bD\oplus \bE,\quad \bD=\bigoplus\bZ/p^{\infty},\quad \bE\text{ {\it reduced}},
  \end{equation}
  meaning (\cite[Definition preceding Theorem 4]{kap_infab}, \cite[following Lemma 4.1.3]{fuchs_abgp}) that $\bE$ has no divisible summands. The hypothesis is that  
  \begin{equation*}
    \sup_{\text{finite }\bF\le \widehat{\bA}}
    \min\left\{|m|\quad : \quad m\in \bZ,\ \widehat{\varphi}|_{\bF}=\text{scaling by $m$}\right\}=\infty
  \end{equation*}
  That supremum will also be infinite when $\bF$ ranges over a {\it countable} set of finite subgroups, so there is no loss in assuming $\widehat{\bA}$ countable to begin with (as we henceforth will). $\bE$ and the set of summands in \Cref{eq:divdec}, then, will be countable.
  
  Because primary cyclic groups and the $\bZ/p^{\infty}$ are \cite[Theorem 10]{kap_infab} the only indecomposable torsion abelian groups, the claim is that if
  \begin{equation*}
    \widehat{\varphi}
    \in
    \End(\text{countable abelian group }\widehat{\bA})
  \end{equation*}
  scales by a proper profinite integer, then it already does so on a direct sum of indecomposable subgroups of $\widehat{\bA}$. The summands $\bD$ and $\bE$ also decompose \cite[Theorem 1]{kap_infab} as the direct sums of their respective {\it primary components} $\bD_p$ and $\bE_p$, i.e. their maximal primary subgroups:
  \begin{equation*}
    \begin{aligned}
      \bD&=\bigoplus_{\text{primes p}}\bD_p
           ,\quad \bD_p=(\bZ/p^{\infty})^{\oplus \alpha_p},\quad 0\le \alpha_p\le \aleph_0\\
      \bE&=\bigoplus_{\text{primes p}}\bE_p
           ,\quad \bE_p\text{ reduced, countable, $p$-primary}. 
    \end{aligned}    
  \end{equation*}
  Recall \cite[\S 9]{kap_infab} the notion of {\it height}
  \begin{equation*}
    h(x) = h_{\bE_p}(x):=\sup\left\{d\in \bZ_{>0}\ |\ \exists y\in \bE_p\text{ with }p^d y=x\right\}
  \end{equation*}
  for an element $x$ of a primary group (such as $\bE_p$, in this case). We will now construct a subgroup
  \begin{equation}\label{eq:dirsumine}
    \bE_p':=\bigoplus_{i}\bZ/p^i\le \bE_p
  \end{equation}
  branching over two cases:
  \begin{enumerate}[(I)]
  \item\label{item:bddh} the heights of $x\in \bE_p$ with $px=0$ are bounded;
  \item\label{item:unbddh} or not.
  \end{enumerate}
  Either way, we construct \Cref{eq:dirsumine} recursively, choosing one summand at a time as in the proof of \cite[Theorem 9]{kap_infab}. In case \Cref{item:unbddh} we can choose countably many summands of unbounded orders and stop. In case \Cref{item:bddh}, we continue the recursion transfinitely, taking unions for limit ordinals. Every step will produce {\it pure} \cite[\S 7]{kap_infab} subgroups of $\bE_p$, which, being by assumption of bounded order, are again summands of $\bE_p$ \cite[Theorem 7]{kap_infab}. But this means that the recursion can proceed until we have exhausted $\bE_p$, so that {\it it} will be $\bE'_p$.   
  
  The subgroup
  \begin{equation*}
    \bigoplus_{\text{primes }p}\left(\bD_p\oplus \bE'_p\right)\le \widehat{\bA},
  \end{equation*}
  by construction a direct sum of indecomposable groups, is now sufficient to distinguish between proper profinite integers and plain integers. 
\end{proof}

\pf{th:abendo}
\begin{th:abendo}
  Once more, the extensibility of {\it integer}-scaling maps is not at issue. Nor is the corresponding direction in \Cref{item:abendo-prof} significantly more difficult: morphisms from a profinite group into a Lie group factor through a finite quotient, on which scaling by $n\in\widehat{\bZ}$ induces scaling by an integer; that then extends, etc. The rest of the proof is devoted, then, to the interesting implication(s). 
  
  Assume for the moment that $\varphi$ extends along morphisms into finite-dimensional tori, a hypothesis common to all items. \Cref{le:presallquot} then implies that $\varphi$ induces an endomorphism $\varphi_{\pi}$ on every quotient
  \begin{equation*}
    \bA\xrightarrowdbl{\quad\pi\quad}\bB:=\bT^d\times\bF,\quad \bF\text{ finite},
  \end{equation*}
  while \Cref{le:torendo} further implies that $\varphi_{\pi}$ scales $\bT^d$ by some $n\in \bZ$. At this point, the proofs of the various items diverge (with \Cref{item:abendo-all} most conveniently saved for last).

  \begin{enumerate}[label={}, leftmargin=*, wide=0pt]
  \item {\bf \Cref{item:abendo-tor}:} I claim first that if $d>0$, then $\varphi_{\pi}$ is scaling by $n$ (globally, so on $\bF$ as well as $\bT^d$).

    Indeed, we have just concluded that all closed subgroups of $\bB=\bT^d\times \bF$ are $\varphi_{\pi}$-invariant, which already implies that $\varphi_{\pi}$ restricted to $\bF\subset \bB$ scaling by some $n'\in \bZ$. If
    \begin{equation*}
      y\in \bF,\quad x\in \bT^d,\quad\mathrm{ord}(y)\ |\ \mathrm{ord}(x) 
    \end{equation*}
    then the finite subgroup of $\bB$ generated by $(x,y)$ has no non-trivial intersection with the right-hand factor $\bF$. But then
    \begin{equation*}
      n(x,y)=(nx,ny)\quad\text{and}\quad(nx,n'y) = \varphi_{\pi}(x,y)
    \end{equation*}
    cannot both belong to the group generated by $(x,y)$ unless $n'y=ny$.

  \item {\bf \Cref{item:abendo-prof}:} We have already established that $\varphi$ induces an endomorphism $\varphi_\bF$ on every finite quotient $\bA\xrightarrowdbl{\pi_{\bF}}\bF$, and part \Cref{item:abendo-tor} then ensures that that endomorphism is multiplication by some $n_{\bF}\in \bZ$. Now, $(n_{\bF})_{\bF}\subset \bZ\subset \widehat{\bZ}$ constitutes a {\it net} \cite[Definition 11.1]{will_top} if the finite quotients are ordered by
    \begin{equation*}
      \pi_{\bF}\le \pi_{\bF'}
      \iff
      \text{the latter factors through the former},
    \end{equation*}
    which, $\widehat{\bZ}$ being compact, has a cluster point \cite[Theorem 17.4]{will_top} $n\in\widehat{\bZ}$; that $n$ will do.     
  \item {\bf \Cref{item:abendo-protor}:} Since part \Cref{item:abendo-tor} supplies the conclusion even under the weaker hypothesis of extensibility along morphisms into {\it tori} unless $\bA$ is profinite, only that case is left. Furthermore, we may then assume by \Cref{item:abendo-prof} that the map scales by some $n\in\widehat{\bZ}$. The hypothesis can also be phrased as the extensibility of $\widehat{\varphi}\in\End(\widehat{\bA})$ along every morphism $\bD\to \widehat{\bA}$ from a discrete torsion-free group $\bD$ (proof of \Cref{le:presallquot}). In that context, \Cref{pr:red2dirsums} reduces the problem to the case when $\widehat{\bA}$ is a countable sum of indecomposable torsion abelian groups.

    Dualizing, the two properties listed in \Cref{item:abendo-protor} translate \cite[Corollary 8.5]{hm4} to the following. 

    \begin{enumerate}[(I), leftmargin=*, wide=0pt]
    \item {\bf : $\widehat{\bA}$ is not reduced.} That is, it has at least one summand $\bZ/p^{\infty}$. Note that the action of $\bZ_p$ (and {\it a fortiori} also that of $\bZ\subset \bZ_p$) is faithful on $\bZ/p^{\infty}$, so an integer scaling a subgroup of the form
      \begin{equation}\label{eq:pinfq}
        \bZ/p^{\infty}\oplus \bZ/q^{\ell}\le \widehat{\bA}
        ,\quad q\text{ prime}
      \end{equation}
      is uniquely determined. It follows that it is enough to work with only such groups.

      If $q\ne p$ then \Cref{eq:pinfq} is the quotient of
      \begin{equation*}
        \bD:=
        \varinjlim\left(
          \bZ
          \lhook\joinrel\xrightarrow{\ q^{\ell}p\ }
          \bZ
          \lhook\joinrel\xrightarrow{\ p\ }
          \bZ
          \lhook\joinrel\xrightarrow{\ p\ }
          \cdots
        \right),
      \end{equation*}
      by the leftmost copy of $\bZ$. The endomorphism group of $\bD$ is the {\it localization} \cite[Chapter 3]{am_comm} of $\bZ$ with respect to the multiplicative set $\{p^n\ |\ n\in \bZ_{\ge 0}\}$; or: the rationals with $p$-power denominators. Since the only such rationals that descend to an endomorphisms of $\bZ_p$ are integers, we are done.

      If, on the other hand, $q=p$, then the argument can proceed as in the proof of part \Cref{item:abendo-tor} above, with the two factors $\bZ/p^{\infty}$ and $\bZ/p^{\ell}$ in place of $\bT^d$ and $\bF$ respectively (for we can then find $y\in \bZ/p^{\ell}$ and $x\in \bZ/p^{\infty}$ with the former's order dividing the latter's, etc.).
      
      We henceforth assume $\widehat{\bA}$ reduced. 
      
    \item {\bf : $\widehat{\bA}$ has infinitely many primary components.} The proof is very similar in spirit to the preceding argument. An 
      \begin{equation}\label{eq:infprod}
        \text{infinite product of the form }\prod_{\text{infinitely many distinct }p}\bZ/p^{\ell_p}\le\widehat{\bA}
      \end{equation}
      is a quotient (by the leftmost copy of $\bZ$) of 
      \begin{equation*}
        \varinjlim\left(
          \bZ
          \lhook\joinrel\xrightarrow{\ p^{\ell_p}\ }
          \bZ
          \lhook\joinrel\xrightarrow{\ (p')^{\ell'}\ }
          \bZ
          \lhook\joinrel\xrightarrow{\ (p'')^{\ell''}\ }
          \cdots
        \right),
      \end{equation*}
      whose endomorphism ring is $\bZ$. It follows that $\widehat{\varphi}$ scales every \Cref{eq:infprod} by an integer; since $\bZ$ acts faithfully thereon, said integers all coincide with, say, some $n\in \bZ$. But then $\widehat{\varphi}$ scales {\it every} sum of indecomposable abelian groups by $n$, including those with infinitely many $p$-primary summands for single primes $p$. The conclusion follows from \Cref{pr:red2dirsums}.
    \end{enumerate}

  \item {\bf \Cref{item:abendo-all}:} The preceding cases (and \Cref{pr:red2dirsums}) have reduced the issue to endomorphisms of subgroups
    \begin{equation*}
      \bD\cong\bigoplus_{\text{prime powers }q}\bZ/q\le \widehat{\bA},
    \end{equation*}
    extensible along morphisms $\bE\to \bD$ in the sense dual to \Cref{def:ext} (as in the proof of \Cref{le:presallquot}). To settle that case, consider the extension
    \begin{equation}\label{eq:zed}
      0\to \bZ\xrightarrow{\quad}\bE\xrightarrow{\quad}\bD\to 0
    \end{equation}
    corresponding \cite[\S 9.1]{fuchs_abgp} the an element of
    \begin{equation*}
      \Ext(\bD,\bZ)\cong \Hom(\bD,\bQ/\bZ)
      \quad\text{\cite[Corollary 9.3.6]{fuchs_abgp}}
    \end{equation*}
    that identifies every summand $\bZ/q$ of $\bD$ isomorphically (in any fashion whatever) with the unique order-$q$ cyclic subgroup of $\bQ/\bZ$. Because the endomorphism ring of the $\bZ$ kernel in \Cref{eq:zed} is $\bZ$, any endomorphism of $\bE$ extending a $\widehat{\bZ}$-scaling of $\bD$ must in fact be a $\bZ$-scaling.  \qedhere
  \end{enumerate}
\end{th:abendo}

\begin{remark}
  The group $\bE$ of \Cref{eq:zed} will generally {\it not} be torsion-free, so the above argument does indeed use the full force of the hypothesis of \Cref{th:abendo}\Cref{item:abendo-all}: the compact abelian group $\widehat{\bE}$ will not, typically, be connected. 
\end{remark}



\addcontentsline{toc}{section}{References}

\Addresses

\end{document}